\documentclass[amsfonts,12pt]{amsart}
\usepackage{cite}
\usepackage{epsfig}
\usepackage{mathtools}
\usepackage[subnum]{cases}
\usepackage{showkeys}
\usepackage[bookmarks=true,
bookmarksnumbered=true, breaklinks=true,
pdfstartview=FitH, hyperfigures=false,
plainpages=false, naturalnames=true,
colorlinks=true,pagebackref=true,
pdfpagelabels]{hyperref}
\usepackage{amsmath,amssymb}
\allowdisplaybreaks
\newtheorem{theorem}{Theorem}[section]

\newtheorem{lemma}[theorem]{Lemma}

\newtheorem{proposition}[theorem]{Proposition}

\begin{document}
	\title{Affine fractional $L_{p}$  P\'olya-Szeg\"o inequalities}
	\author[Youjiang Lin, Jiaming Lan, Jinghong Zhou]
	{Youjiang Lin, Jiaming Lan, Jinghong Zhou}
	\address{School of Mathematics and Statistics, Chongqing Technology and Business University, Chongqing,
400067, China} \email{yjl@ctbu.edu.cn}
	\address{School of Mathematics and Statistics, Chongqing Technology and Business University, Chongqing,
400067, China} \email{lanjiaming@ctbu.edu.cn}
    \address{School of Mathematics and Statistics, Chongqing Technology and Business University, Chongqing,
400067, China} \email{zjh@ctbu.edu.cn}
     
	\thanks{This work is supported in part by funds from the National Natural Science Foundation of China NSFC 12371137 and NSFC 11971080.}
	\maketitle
	\pagestyle{myheadings}
	\markboth{Youjiang Lin et al.}{Affine fractional $L_{p}$  PS inequalities}
	
	\begin{abstract}
 Affine fractional $L_{p}$  P\'olya-Szeg\"o inequalities for two functions on $\mathbb{R}^n$ are established, which are stronger than the Euclidean  fractional $L_{p}$  P\'olya-Szeg\"o inequalities.
	\end{abstract}

\maketitle
\section{Introduction}
	The classical P\'olya-Szeg\"o principle \cite{PS51} asserts that the $L^p$ norm of the gradient of a function defined on $\mathbb{R}^n$ is non-increasing under symmetric decreasing rearrangement. The profound result serves as a cornerstone in solving numerous variational problems across analysis, particularly in establishing optimal forms of isoperimetric inequalities, deriving sharp constants in Sobolev embeddings, and obtaining precise a priori estimates for solutions to second-order elliptic and parabolic boundary value problems. The principle's far-reaching implications make it an indispensable tool in modern PDE theory and geometric analysis (see, e.g.,  \cite{TA76, C00, CF06, FV04}).
	
	A full affine counterpart to the classical P\'olya-Szeg\"o principle was developed by A. Cianchi, E. Lutwak, D. Yang and G. Zhang \cite{CLYZ09}, extending fundamental results from G. Zhang \cite{Zhang99} and E. Lutwak, D. Yang and G. Zhang \cite{LYZ02}. In the affine P\'olya-Szeg\"o inequality, the $L^p$ norm of the Euclidean length of the gradient is replaced by an affine invariant of functions, the $L^p$ affine energy, leading to an inequality which is significantly stronger than its classical Euclidean counterpart.
	In addtion, C. Haberl, F.E. Schuster and J. Xiao \cite{CF12} established a notable asymmetric form of the affine P\'olya-Szeg\"o inequality, enhancing  the affine P\'olya-Szeg\"o principle due to Cianchi et al \cite{CLYZ09}. G. Talenti \cite{GT94} establish a Euclidean Orlicz P\'olya-Szeg\"o principle, which extended the classical P\'olya-Szeg\"o principle to Orlicz-Sobolev spaces. Y. Lin \cite{Lin17} prove an affine Orlicz P\'olya-Szeg\"o principle for $\log$-concave functions by functional Steiner symmetrizations. In recent years, many important generalizations and variations of affine P\'olya-Szeg\"o principle have been obtained (see, e.g., \cite{GZ98, CF09,CF09-, HJM16, Lutwak00, N16, Wang12,Wang13}).
	
	F.J. Almgren and E.H. Lieb \cite[Theorem 9.2]{AG89} established Euclidean fractional $L_p$ P\'olya-Szeg\"o inequality in the fractional Sobolev space:
\begin{eqnarray}\label{e1.6}
\int_{\mathbb{R}^n}\int_{\mathbb{R}^n}\frac{{(f\left ( x \right )-f\left ( y \right ))}^{p}}{{|x-y|}^{n+ps}}dxdy\ge\int_{\mathbb{R}^n}\int_{\mathbb{R}^n}\frac{{(f^{\star}\left ( x \right )-f^{\star}\left ( y \right ))}^{p}}{{|x-y|}^{n+ps}}dxdy
\end{eqnarray}
for $0<s<1$, $p\ge 1$ and $f\in W^{s,p}\left ( \mathbb{R}^{n} \right )$, where $f^{\star}$ denotes the symmetric decreasing rearrangement and  $W^{s,p}\left ( \mathbb{R}^{n} \right )$ denotes the fractional Sobolev space of $L_p$ functions (see Section {\ref{s2}} for detailed definitions). J. Haddad and M. Ludwig \cite{JM24} established anisotropic Euclidean fractional P\'olya-Szeg\"o inequalities for fractional $L_p$ Sobolev norms: If $f\in L^{p}\left ( \mathbb{R}^{n}  \right )$ is non-negative and $K\subset \mathbb{R}^{n}  $ a star body, then
\begin{eqnarray}\label{122}
	\int_{\mathbb{R}^{n} }\int_{\mathbb{R}^{n} }\frac{\left | f\left ( x \right )-f\left ( y \right ) \right | ^{p}    }{\left \| x-y \right \|^{n+ps} _{K}  }dxdy\ge\int_{\mathbb{R}^{n} }\int_{\mathbb{R}^{n} }\frac{\left | f^{\star } \left ( x \right )-f^{\star } \left ( y \right ) \right | ^{p}   }{\left \| x-y \right \|^{n+ps} _{K^{\star } }  }dxdy .
\end{eqnarray}

 The anisotropic $s$-seminorms, i.e., the left side of (\ref{122}), introduced  by M. Ludwig \cite{Ludwig14}, reflect a fine structure of the anisotropic fractional Sobolev spaces. She established that
 \begin{eqnarray}\label{e133}
 \lim _{s \rightarrow 1^{-}}(1-s) \int_{\mathbb{R}^n}\int_{ \mathbb{R}^n} \frac{|f(x)-f(y)|^p}{\|x-y\|_K^{n+p s}} d x d y=\frac{2}{p} \int_{\mathbb{R}^n}\|\nabla f(x)\|_{Z_p^* K}^p d x
 \end{eqnarray}
 for $f \in W^{1, p}\left(\mathbb{R}^n\right)$ with compact support, where the norm associated with $Z_p^* K$, the polar $L_p$ moment body of $K$, is defined as
 \begin{eqnarray}
  \|v\|_{Z_p^* K}^p=\frac{n+p}{2} \int_K|v \cdot x|^p d x
 \end{eqnarray}
 for $v \in \mathbb{R}^n$ and a convex body $K \subset \mathbb{R}^n$.
 Later, D. Ma \cite{Ma14} proved the asymmetric version of (\ref{e133}).
 
In the remarkable paper \cite{JM24}, J. Haddad and M. Ludwig also  obtained affine fractional $L_p$ P\'olya-Szeg\"o inequalities:
\begin{eqnarray}\label{1.7}
&&\int_{\mathbb{R}^n}\int_{\mathbb{R}^n}\frac{{(f\left ( x \right )-f\left ( y \right ))} ^{p}}{{|x-y|}^{n+ps}  }dxdy\nonumber\\&\ge&
n\omega _{n}^{\frac{n+ps}{n} }\left ( \frac{1}{n}\int_{\mathbb{S}^{n-1}  }\left ( \int_{0}^{\infty }t^{ps-1}\int_{\mathbb{R}^{n}  }\left | f\left ( x+t\xi  \right )-f\left ( x \right )   \right |^{p} dxdt  \right )^{-\frac{n}{ps} }d\xi       \right )^{-\frac{ps}{n} }\nonumber\\&\ge& \int_{\mathbb{R}^n}\int_{\mathbb{R}^n}\frac{{(f^{\star}\left ( x \right )-f^{\star}\left ( y \right ))}  ^{p}}{{|x-y|}^{n+ps}}dxdy
,
\end{eqnarray}
for  $0<s<1$, $1<p<n/s$ and $f\in W^{s,p}\left ( \mathbb{R}^{n} \right )$. That is significantly stronger than  the  Euclidean fractional $L_p$ P\'olya-Szeg\"o inequality (\ref{e1.6}). More papers on fractional inequalities, see, e.g., \cite{JM25, K21, Ma14}.

The paper aims to establish affine fractional $L_{p}$  P\'olya-Szeg\"o inequalities  on two different functions, which generalize the affine fractional $L_{p}$ P\'olya-Szeg\"o inequalities (\ref{1.7}). For $0 < s <1$ and $p\ge1$, we define the generalized fractional Sobolev space ${W}^{s,p}(\mathbb{R}^{n},\; \mathbb{R}^{2})$ assoicated with $f$ and $h$ as
\begin{eqnarray*}
	{W}^{s,p}(\mathbb{R}^{n}, \mathbb{R}^{2})= \left\{(f,h):f,\;h\in{L}^{p}(\mathbb{R}^{n}),\;\int_{\mathbb{R}^{n}}\int_{\mathbb{R}^{n}}\frac{{|f(x)-h(y)|}^{p}}{{| x-y|}^{n+ps}}dxdy<\infty \right\}.
\end{eqnarray*}
\begin{theorem}\label{1.a} Let $0 < s < 1$ and $1 < p < n/s$. For non-negative functions $(f,h)\in{W}^{s,p}(\mathbb{R}^{n},\mathbb{R}^{2})$ and $f\in {W}^{s,p}(\mathbb{R}^{n})$,
	\begin{eqnarray}\label{1.e}
		&&\int_{\mathbb{R}^n}\int_{\mathbb{R}^n}\frac{{(f\left ( x \right )-h\left ( y \right ))} ^{p}}{{|x-y|}^{n+ps}  }dxdy\nonumber\\
		&\ge&
		n\omega _{n}^{\frac{n+ps}{n} }\left ( \frac{1}{n}\int_{\mathbb{S}^{n-1}  }\left ( \int_{0}^{\infty }t^{ps-1}\int_{\mathbb{R}^{n}  }\left | f\left ( x+t\xi  \right )-h\left ( x \right )   \right |^{p} dxdt  \right )^{-\frac{n}{ps} }d\xi  \right )^{-\frac{ps}{n} }\nonumber\\&\ge& \int_{\mathbb{R}^n}\int_{\mathbb{R}^n}\frac{{(f^{\star}\left ( x \right )-h^{\star}\left ( y \right ))}  ^{p}}{{|x-y|}^{n+ps}}dxdy.   
	\end{eqnarray}
	 There is equality in the first inequality if $f$, $h$ are radially symmetric.There is equality in the second inequality if and only if $f= f^{\star }(\phi x+x_{0} )$, $h= h^{\star }(\phi x+x_{0} )$ for some $\phi \in GL(n)$ and $x_{0}\in \mathbb{R}^n$.
\end{theorem}

In order to prove Theorem \ref{1.a}, we define the {\it generalized fractional $L_p$ polar projection body} ${\Pi}_{p}^{*,s}(f,h)$ associated with $f$ and $h$, defined as the  star-shaped set whose  gauge function for $\xi\in \mathbb{S}^{n-1}$,
\begin{eqnarray}
	{\left \| \xi \right \|}_{{\Pi}_{p}^{*,s}(f,h)}^{ps}=\int_{0}^{\infty }{t }^{-ps-1} \int_{\mathbb{R}^{n}}{|f(x+t\xi) -h(x)|}^{p}dxdt.
\end{eqnarray}
The  second inequality in (\ref{1.e}) now can be written as
\begin{eqnarray}\label{1.g}
			n{\omega} _{n}^{\frac{n+ps}{n}} |{\Pi }_{p}^{*,s}(f,h)|^{-\frac{ps}{n}}\ge\int_{\mathbb{R}^n}\int_{\mathbb{R}^n}\frac{{(f^{\star}\left ( x \right )-h^{\star}\left ( y \right ))}  ^{p}}{{|x-y|}^{n+ps}}dxdy.
\end{eqnarray}
We note that both sides of (\ref{1.g}) are translation invariant with respect to $f$ and $h$, and for volume-preserving linear transformations $\phi :\mathbb{R}^{n}\to \mathbb{R}^{n} $,
\begin{eqnarray}
	{\textstyle \Pi_{p}^{\ast ,s}}\left ( f\circ \phi ^{-1}, h\circ \phi ^{-1} \right )=\phi  {\textstyle \Pi_{p}^{\ast ,s}}\left ( f, h \right ),
\end{eqnarray}
it implies that (\ref{1.g}) is a $SL(n)$ affine inequality. Moreover, for $r>0$, we have
\begin{eqnarray}
|{\Pi }_{p}^{*,s}(f\circ r,h\circ r)|^{-\frac{ps}{n}}=r^{-n+ps}|{\Pi }_{p}^{*,s}(f,h)|^{-\frac{ps}{n}},
\end{eqnarray}
\begin{eqnarray}
 \int_{\mathbb{R}^n}\int_{\mathbb{R}^n}\frac{{(f^{\star}\left ( rx \right )-h^{\star}\left ( ry \right ))}  ^{p}}{{|x-y|}^{n+ps}}dxdy=r^{-n+ps} \int_{\mathbb{R}^n}\int_{\mathbb{R}^n}\frac{{(f^{\star}\left ( x \right )-h^{\star}\left ( y \right ))}  ^{p}}{{|x-y|}^{n+ps}}dxdy.
\end{eqnarray}
They imply that (\ref{1.g}) is  a $GL(n)$ affine inequality.

This paper is organized as follows. In Section \ref{s2}, we  establish notations and list some basic facts of star-shaped sets, dual mixed volumes, functional Steiner symmetrizations and  Sobolev space. In Section \ref{s3}, we prove the generalized fractional $L_p$ polar projection body is a star body with the origin in its interior.
In Section \ref{s4}, we define generalized asymmetric fractoinal  $L_p$ polar projection body and prove the $L_p$ polar projection body is a star body with the origin in its interior.
In Section \ref{s5}, we establish generalized anisotropic fractional $L_p$ P\'olya-Szeg\"o inequality and its asymmetric counterpart. In fact, we prove that the dual mixed volume $\widetilde{V}_{-ps}\left ( K, {\textstyle \Pi_{p}^{\ast ,s}}\left ( f,h \right ) \right )$ is decreasing by symmetric decreasing rearrangements.
In Section \ref{s6}, we prove the main theorem, i.e., the affine fractonal $L_{p}$ P\'olya-Szeg\"o inequality on two different functions and its asymmetric counterpart.
	
	\section{Preliminaries}\label{s2}
	Let $\mathbb{R}^n$ denote $n$-dimensional Euclidean space with canonical inner product $x \cdot y$, for $x, y \in$ $\mathbb{R}^n$;
	throughout we assume that $n$ is a natural number. Let $o$ denote the origin of $\mathbb{R}^n.$
	Write $\|x\|=\sqrt{x \cdot x}$ for the norm of $x\in\mathbb{R}^n$. Let $\mathbb{S}^{n-1}$ denote the unit sphere in $\mathbb{R}^n$. Let $B^n$ denote the unit ball centered at origin. The group of  linear transformations of $\mathbb{R}^n$ is denoted by $\operatorname{GL}(n)$. The group of special linear transformations of $\mathbb{R}^n$ is denoted by $\operatorname{SL}(n)$. We will write $|K|$ rather than $V(K)$ to denote $n$-dimensional volume of $K\subset\mathbb{R}^{n}$.
	\subsection{Star-shaped sets and dual mixed volumes}
	For quick reference, we list some basic facts about star-shaped sets and dual mixed volumes. For more information, see R.J. Gardner \cite{Gardner06}, R. Schneider \cite{Schneider14} and E. Lutwak \cite{Lutwak75}.
	
	A closed set $K\subseteq \mathbb{R}^{n}  $ is regarded as {\it star-shaped} (with respect to the origin) if the interval $\left [ 0,x \right ]\subset K $ for every $x\in K$. The {\it radial function} ${\rho}_{K}:\mathbb{R}^n \setminus \{0\} \to [0,\infty ]$ is defined by
	\begin{eqnarray}
		{\rho}_{K}(x)=\sup\{ \lambda \geq 0:\lambda x\in K\}\nonumber
	\end{eqnarray}
	and the function $\left \| \cdot  \right \|_{K}:\mathbb{R}^{n}\to \left [ 0,\infty  \right ] $ of a  star-shaped set defined as 
	\begin{eqnarray}
		{||x ||}_{K}=\inf\{ \lambda >0:x\in\lambda K\}.\nonumber
	\end{eqnarray}
	is called the {\it gauge function} of $K$.
	
	For a star-shaped set $K$ in $\mathbb{R}^n$ whose radial function is measurable, its $n$-dimensional Lebesgue measure or volume is given by
	\begin{equation}
		\left|K\right|=\frac{1}{n}\int_{\mathbb{S}^{n-1}}\rho^{n} _{K}\left ( \xi  \right )d\xi. 
	\end{equation}
	We call a  star-shaped set $K \subset \mathbb{R}^n$ a {\it star body} if its  radial function is  positive and locally Lipschitz continuous in  $\mathbb{R}^{n}\setminus \{0\} $. On the set of star bodies, the {\it $q$-radial sum} $K\tilde{+}_{q}L$ for $q\ne 0$ of $K,L\subset \mathbb{R}^{n}$ is defined by
	\begin{eqnarray}\label{e2.2}
		\rho ^{q}\left ( K\tilde{+}_{q}L,\xi    \right )= \rho ^{q}\left ( K,\xi    \right )+\rho ^{q}\left ( L,\xi \right )
	\end{eqnarray}
	for $\xi\in \mathbb{S}^{n-1} $ (cf. \cite[Section 9.3]{Schneider14}). The {\it dual Brunn-Minkowski inequality} (cf. \cite [(9.41)]{Schneider14}) states that for star bodies $K,L\subset \mathbb{R}^{n}$ and $q>0$, 
	\begin{eqnarray}\label{2.3}
		\left | K\tilde{+}_{-q}L   \right |^{-q/n}\ge \left | K \right |^{-q/n}+\left | L \right |^{-q/n},
	\end{eqnarray}
	with equality precisely if $K$ and $L$ are dilates, that is, $\lambda >0$ such that $K=\lambda L$.
	
	Let $\alpha \in \mathbb{R} \setminus \{0, n\}$. For  star bodies  $K, L \subset \mathbb{R}^n$, the {\it dual mixed volume} is defined as
	\begin{eqnarray}
		\widetilde{V}_{\alpha}\left ( K,L \right )=\frac{1}{n}\int_{\mathbb{S}^{n-1} }\rho_{K }\left ( \xi  \right ) ^{n-\alpha }\rho_{L}\left ( \xi  \right )^{\alpha }d\xi.\nonumber
	\end{eqnarray}
	Note that $\widetilde{V}_{\alpha}\left (K,K\right)=|K|$ and that 
	\begin{eqnarray}
		\widetilde{V}_{\alpha } \left ( K, L_{1}\tilde{+}_{\alpha } L_{2}    \right )= \widetilde{V}_{\alpha }\left ( K,L_{1}  \right )+\widetilde{V}_{\alpha }\left ( K,L_{2}  \right )
	\end{eqnarray}
	for  star bodies $K,L_{1},L_{2} \subseteq \mathbb{R}^{n} $.
	
	For $\alpha<0$ or $\alpha>n$, the {\it dual mixed volume inequality} states that 
	\begin{eqnarray}\label{2a}
		\widetilde{V}_{\alpha} (K,L)\ge |K|^{n-\alpha/n}|L|^{\alpha/n},
	\end{eqnarray}
	and the reverse inequality
	\begin{eqnarray}\label{2aa}
		\widetilde{V}_{\alpha} (K,L)\le |K|^{n-\alpha/n}|L|^{\alpha/n}
	\end{eqnarray}
	 holds for $0<\alpha<n$. 
	Equalities in (\ref{2a}) and (\ref{2aa}) hold if and only if $K$ and $L$ are dilates, in other words, $\rho _{K}=c \rho _{L}$ almost everywhere on $\mathbb{S}^{n-1}$ for some $c>0$. The definition of dual mixed volumes for star bodies is attributed to E. Lutwak \cite{Lutwak75}, the dual mixed volume inequality is deduced from H\"older's inequality as well (see \cite[Section 9.3]{Schneider14} or \cite[B.29]{Gardner06}).
	\subsection{Function spaces}\label{s2.2}
	In the section we summarize the necessary definitions about Sobolev space. For addtional details,  the reader could consult the book of Maz'ya \cite{Mazya11} and Adams \cite {Adams03}.
	
	For $p \ge 1$ and measurable $f : \mathbb{R}^{n}\to\mathbb{R}$, let
	\begin{eqnarray*}
		{\left \| f \right \|}_{p}= {\left(\int_{\mathbb{R}^{n}}{|f(x)|}^{p}dx\right)}^{1/p}.
	\end{eqnarray*}
	We set the {\it super-level sets} $\{f \ge t\}=\{x \in \mathbb{R}^{n} : f(x) \ge t\}$ for $t \in \mathbb{R}$. A function $f$ is called non-zero when $\{ f\neq 0\}$ has postive measure, with functions being treated as equivalent when they concide expect on a null set. For $p \ge 1$, let
	\begin{eqnarray*}
		{L}^{p}(\mathbb{R}^{n})= \left\{f:\mathbb{R}^{n}\to\mathbb{R}:f \;{\rm is\; measurable}, {\left \| f \right \|}_{p}<\infty \right\}.
	\end{eqnarray*}
In this context, measurability is always defined in terms of the standard Lebesgue measure on $\mathbb{R}^{n}$.
	
	For $0 < s <1$ and $p\ge1$, the fractional Sobolev space ${W}^{s,p}(\mathbb{R}^{n})$ (see  \cite[Section 9.1]{AG89}) is defined as
	\begin{eqnarray}\label{2.2a}
		{W}^{s,p}(\mathbb{R}^{n})= \left\{f\in{L}^{p}(\mathbb{R}^{n}):\int_{\mathbb{R}^{n}}\int_{\mathbb{R}^{n}}\frac{{|f(x)-f(y)|}^{p}}{{| x-y|}^{n+ps}}dxdy<\infty \right\}.
	\end{eqnarray}
	For $p\ge1$, we set
	\begin{eqnarray}
		{W}^{1,p}(\mathbb{R}^{n})= \left\{f\in{L}^{p}(\mathbb{R}^{n}):|\nabla f|\in{L}^{p}(\mathbb{R}^{n}) \right\},
	\end{eqnarray}
	where $\nabla f$ is the weak gradient of $f$.
	
	The generalized fractional Sobolev space ${W}^{s,p}(\mathbb{R}^{n},\; \mathbb{R}^{2})$ concerning two different functions is defined as
	\begin{eqnarray}\label{e1.4}
		{W}^{s,p}(\mathbb{R}^{n}, \mathbb{R}^{2})= \left\{(f,h):f,\;h\in{L}^{p}(\mathbb{R}^{n}),\;\int_{\mathbb{R}^{n}}\int_{\mathbb{R}^{n}}\frac{{|f(x)-h(y)|}^{p}}{{| x-y|}^{n+ps}}dxdy<\infty \right\}.
	\end{eqnarray}
By (\ref{2.2a}) and (\ref{e1.4}), if $(f,f)\in {W}^{s,p}(\mathbb{R}^{n},\; \mathbb{R}^{2})$, then $f\in {W}^{s,p}(\mathbb{R}^{n})$.
	\subsection{Symmetrization}
	For a set $E\subset \mathbb{R}^n $, the characteristic function $1_{E} $ is denoted by $1_{E}\left ( x \right )=1  $ for $x\in E$ and $1_{E}\left ( x \right )=0  $ otherwise. Let $E\subset \mathbb{R}^n $ be a Borel set of finite measure. The Schwarz symmetral of $E$, defined by $E^{\star }$, is a closed, centered Euclidean ball whose volume agrees with that of $E$, i.e.,
	\begin{eqnarray}
	E^{\star }=\left \{ x\in \mathbb{R}^{n}: \omega _{n}\left | x \right |^{n}\le \left | E \right |\right \}, \nonumber
	\end{eqnarray}
	where $\omega _{n}=\pi ^{\frac{n}{2}}/\Gamma \left ( 1+\frac{n}{2}\right )$ is $n$-dimensional  volume enclosed by the unit sphere $\mathbb{S}^{n-1}$ and $\Gamma$ is the gamma function.
	
	Let $f:\mathbb{R}^n \to \mathbb{R}$ be a non-negative measurable function with super-level sets $
	\left \{f\ge t  \right \}$ of finite measure for any $t>0$. The {\it layer cake formula} states that 
	\begin{eqnarray}\label{2c}
		f\left ( x \right )=\int_{0}^{\infty }1_{\left \{ f\ge t \right \} }\left ( x \right )dt\notag
	\end{eqnarray}
	for almost every $x\in \mathbb{R}^n$. The {\it symmetric decreasing rearrangement} of $f$, denoted by $f^{\star } $, is defined by
	\begin{eqnarray}
		f^{\star }\left ( x \right )=\int_{0}^{\infty }1_{\left \{ f\ge t \right \}^{\star }  }(x)dt   \notag
	\end{eqnarray}
	for $x\in \mathbb{R}^n$. Hence $f^{\star } $ is determined by the properties of being radially symmetric and having super-level sets that are balls of the same measure as the super-level sets of $f$.
	 Our results are built upon the Riese rearrangement inequality, available in full generality, for example, in \cite{BLL74}.
	\begin{theorem} \label{Rri} (Riesz’s rearrangement inequality). For $f,g,k:\mathbb{R}^n\to \mathbb{R}$ non-negative, measurable functions with super-level sets of finite measure,
		\begin{eqnarray*}
			\int_{\mathbb{R}^n}\int_{\mathbb{R}^n}f(x)k(x-y)g(y)dxdy\leq \int_{\mathbb{R}^n}\int_{\mathbb{R}^n}f^{\star }(x)k^{\star}(x-y)g^{\star}(y)dxdy.
		\end{eqnarray*}
	\end{theorem}
	We will apply Burchard's criterion \cite{Burchard96} for determining the cases of equality in the  Riesz rearrangement inequality. 
	\begin{theorem} (Burchard)\label{2.b}. Let $A$, $B$ and $C$ be sets of finite positive measure in $\mathbb{R}^n$ and denote by $\alpha$, $\beta$ and $\gamma$ the radii of their Schwarz symmetrals $A^{\star}$, $B^{\star}$ and $C^{\star}$. For $|\alpha -\beta |<\gamma <\alpha +\beta$, there is equality in
		\begin{eqnarray*}
			\int_{\mathbb{R}^n}\int_{\mathbb{R}^n}1_{A}(y)1_{B}(x-y)1_{C}(x)dxdy\leq \int_{\mathbb{R}^n}\int_{\mathbb{R}^n}1_{A^{\star}}(y)1_{B^{\star}}(x-y)1_{C^{\star}}(x)dxdy
		\end{eqnarray*}
		if and only if, up to sets of measure zero,
		\begin{eqnarray*}
			A=a+\alpha D,\; B=b+\beta D,\; C=c+\gamma D,
		\end{eqnarray*}
		where $D$ is a centered ellipsoid, and $a$, $b$ and $c = a + b$ are vectors in $\mathbb{R}^n$.
	\end{theorem}
	
	\subsection{Anisotropic fractional Sobolev norms}
	J. Haddad and M. Ludwig \cite{JM24} introduced the definition of anisotropic
	fractional ${L}_{p}$ Sobolev norm of function with respect to star bodies, more information for anisotropic fractional Sobolev norms is provided by M. Ludwig \cite{Ludwig14}.
	
	Let $0 < s <1$ and $p\ge1$. For $K \subset \mathbb{R}^{n}$ a star body and $(f,h)\in{W}^{s,p}(\mathbb{R}^{n},\mathbb{R}^{2})$, we define the {\it anisotropic fractional ${L}_{p}$ Sobolev norm} of $f$ and $h$ with respect to $K$ by
	\begin{eqnarray}\label{2.4a}
		\int_{\mathbb{R}^{n}}\int_{\mathbb{R}^{n}}\frac{{|f(x)-h(y)|}^{p}}{{\left \| x-y\right\|}_{K}^{n+ps}}dxdy.
	\end{eqnarray}
	In the case  $f$=$h$, it is consistent with the definition of  anisotropic fractional $L_p$ Sobolev norms introduced in \cite{Ludwig14} for $K$ a convex body (also, see \cite{Ludwig1496}). For $K = {B}^{n}$, the
	Euclidean unit ball, the classical $s$-fractional ${L}_{p}$ Sobolev norm of $f$ is obtained. The
	limit as $s\to{1}^{-}$ was characterized in \cite{JHP01} in the Euclidean setting and in \cite{Ludwig14} in the
	anisotropic case. We will also consider the following asymmetric versions of (\ref{2.4a}),
	\begin{eqnarray}
		\int_{\mathbb{R}^{n}}\int_{\mathbb{R}^{n}}\frac{{(f(x)-h(y))}^{p}_{+}}{{\left \| x-y\right\|}_{K}^{n+ps}}dxdy,\;\;\int_{\mathbb{R}^{n}}\int_{\mathbb{R}^{n}}\frac{{(f(x)-h(y))}^{p}_{-}}{{\left \| x-y\right\|}_{K}^{n+ps}}dxdy,
	\end{eqnarray}
	where ${a_{+} = \max\{a,0\}}$ and ${a_{-} = \max\{-a,0\}}$ for $a \in \mathbb{R}$.
	\section{Generalized fractional ${L}_{p}$ polar projection bodies}\label{s3}
	Let $0 < s < 1$ and $1 < p < n/s$. For  measurable functions $f,h : \mathbb{R}^{n} \to \mathbb{R}$, define the
	generalized $s$-fractional ${L}_{p}$ polar projection body ${\Pi}_{p}^{*,s}(f,h)$ as the star-shaped set given by the gauge function
	\begin{eqnarray}\label{3a}
		{\left \| \xi \right \|}_{{\Pi}_{p}^{*,s}(f,h)}^{ps}=\int_{0}^{\infty }{t }^{-ps-1} \int_{\mathbb{R}^{n}}{|f(x+t\xi) -h(x)|}^{p}dxdt
	\end{eqnarray}
	for $\xi\in\mathbb{R}^{n}\setminus \{0\}$. 
	Since  ${\left \| \cdot \right \|}_{{\Pi}_{p}^{*,s}(f,h)}$ is a one-homogeneous function on $\mathbb{R}^{n}$, we can define  ${\left \| 0 \right \|}_{{\Pi}_{p}^{*,s}(f,h)}=0$.
	Let $K\subset \mathbb{R}^{n}$ be a star body. For $(f,h)\in{W}^{s,p}(\mathbb{R}^{n},\mathbb{R}^{2})$, by spherical coordinates, Fubini's theorem and (\ref{3a}), we have
	\begin{eqnarray}
		\int_{\mathbb{R}^{n}}\int_{\mathbb{R}^{n}}\frac{{|f(x)-h(y)|}^{p}}{{\left \| x-y\right\|}_{K}^{n+ps}}dxdy&=&\int_{\mathbb{R}^{n}}\int_{\mathbb{R}^{n}}\frac{{|f(y+z)-h(y)|}^{p}}{{\left \| z\right \|}_{K }^{n+ps}}dzdy\nonumber\\
		&=&\int_{\mathbb{R}^{n}}\int_{\mathbb{S}^{n-1}}\int_{0}^{\infty }\frac{{|f(y+t\xi )-h(y)|}^{p}{t}^{n-1}}{{\left \| t\xi\right \|}_{K }^{n+ps}}dtd\xi dy\nonumber\\
		&=&\int_{\mathbb{S}^{n-1}}{{\rho}_{K} (\xi)}^{n+ps}\int_{0}^{\infty }{t }^{-ps-1}\int_{\mathbb{R}^{n}}{|f(y+t\xi )-h(y)|}^{p}dydtd\xi\nonumber\\
		&=&\int_{\mathbb{S}^{n-1}}{{\rho}_{K} (\xi)}^{n+ps}{{\rho}_{{\Pi}_{p}^{*,s}(f,h)} (\xi)}^{-ps}d\xi.\label{e34}
	\end{eqnarray}
	Hence,
	\begin{eqnarray}\label{3b}
		\int_{\mathbb{R}^{n}}\int_{\mathbb{R}^{n}}\frac{\left | f\left ( x \right )-h\left ( y \right )   \right |^{p}}{\left \| x-y \right \|^{n+ps} _{K} }dxdy=n\widetilde{V}_{-ps}\left ( K, {\textstyle \Pi_{p}^{\ast ,s}}\left ( f,h \right ) \right ). 
	\end{eqnarray}
	
	Next, we establish basic properties of generalized fractional ${L}_{p}$ polar projection bodies.
	\begin{proposition}\label{3.a}
		For non-zero functions $(f,h)\in{W}^{s,p}(\mathbb{R}^{n},\mathbb{R}^{2})$ and $f\in {W}^{s,p}(\mathbb{R}^{n})$, the set ${\Pi}_{p}^{*,s}(f,h)$ is a star body with the origin in its interior. Moreover, there is $c>0$ depending only on $f,h$ and $p$ such that ${\Pi}_{p}^{*,s}(f,h)\subset c{B}^{n}$ for every $s\in(0,1)$.	 
	\end{proposition}
	\begin{proof}
		Firstly, we show that ${\Pi}_{p}^{*,s}(f,h)$ is bounded. For fixed $N>1$, we take $r>1$ sufficiently large so that ${\left \|f \right \|}_{{L}^{p}(r{B}^{n})} \ge \frac{N-1}{N}{\left \|f \right \|}_{p}$, ${\left \|h \right \|}_{{L}^{p}(r{B}^{n})} \ge \frac{N-1}{N}{\left \|h \right \|}_{p}$ and easily see that for $t>2r$,
		\begin{eqnarray*}
			{\left \|f(\cdot +t\xi)-h(\cdot ) \right \|}_{p}
			&\ge&{\left \|f(\cdot +t\xi)-h(\cdot ) \right \|}_{{L}^{p}(r{B}^{n}-t\xi)}\nonumber\\
			&=& {\left \|f(\cdot )-h(\cdot-t\xi ) \right \|}_{{L}^{p}(r{B}^{n})}\nonumber\\
			&\ge&{\left \|f \right \|}_{{L}^{p}(r{B}^{n})}-{\left \|h(\cdot-t\xi )  \right \|}_{{L}^{p}(r{B}^{n})}\nonumber\\
			&\ge&\frac{N-1}{N}{\left \|f \right \|}_{p}-\frac{1}{N}{\left \|h \right \|}_{p}.\nonumber\\
		\end{eqnarray*} 
		Hence,
		\begin{eqnarray*}
			\int_{0}^{\infty }{t }^{-ps-1} \int_{\mathbb{R}^{n}}{|f(x+t\xi) -h(x)|}^{p}dxdt
			&\ge&\int_{0}^{\infty }{t }^{-ps-1}{\left(\frac{N-1}{N}{\left \|f \right \|}_{p}-\frac{1}{N}{\left \|h \right \|}_{p} \right)}^{p}dt\nonumber\\
			&\ge&{\left(\frac{N-1}{N}{\left \|f \right \|}_{p}-\frac{1}{N}{\left \|h \right \|}_{p}\right) }^{p}\int_{r}^{\infty }{t }^{-ps-1}dt\nonumber\\
			&=&{\left(\frac{N-1}{N}{\left \|f \right \|}_{p}-\frac{1}{N}{\left \|h \right \|}_{p}\right) }^{p}\frac{{r}^{-ps}}{ps}\ge c^{ps} ,
		\end{eqnarray*} 
		which implies that ${\Pi}_{p}^{*,s}(f,h)\subset c{B}^{n}$ for $c>0$ independent of $s$.
		
		Next, we show that ${\Pi}_{p}^{*,s}(f,h)$ contains the origin in its interior. First observe that for $\xi$, $\eta\in{\mathbb{R}}^{n}$, by the triangle inequality, H\"older inequality and a change of variable,
		\begin{eqnarray}\label{3c}
			{\left \|\xi+\eta \right \|}^{ps}_{{\Pi}_{p}^{*,s}(f,h)}
			&=& \int_{0}^{\infty }{t }^{-ps-1}{\left \|f(\cdot +t\xi+t\eta)-h(\cdot ) \right \|}^{p}_{p}dt\nonumber\\
			&=& \int_{0}^{\infty }{t }^{-ps-1}{\left \|f(\cdot +t\xi+t\eta)-f\left ( \cdot +t\xi  \right )+f\left ( \cdot +t\xi  \right ) -h(\cdot ) \right \|}^{p}_{p}dt\nonumber\\
			&\le&\int_{0}^{\infty }{t }^{-ps-1}\left ( {\left \|f(\cdot +t\xi+t\eta)-f\left ( \cdot +t\xi  \right )\right\|_{p} +\left\|f\left ( \cdot +t\xi  \right ) -h(\cdot ) \right \|_{p} } \right )^{p} dt\nonumber\\
			&\le&\int_{0}^{\infty }{t }^{-ps-1}2^{p-1} \left ( {\left \|f(\cdot +t\xi+t\eta)-f\left ( \cdot +t\xi  \right )\right\|^{p} _{p} +\left\|f\left ( \cdot +t\xi  \right ) -h(\cdot ) \right \|^{p} _{p} } \right ) dt\nonumber\\
			&=& {2}^{p-1}\left({\left \|\eta \right \|}^{ps}_{{\Pi}_{p}^{*,s}(f,f)}+{\left \|\xi \right \|}^{ps}_{{\Pi}_{p}^{*,s}(f,h)}\right)\nonumber\\
			&\le&{2}^{p-1}\left({\left \|\eta \right \|}^{ps}_{{\Pi}_{p}^{*,s}(f,f)}+{\left \|\xi  \right \|}^{ps}_{{\Pi}_{p}^{*,s}(f,f)}+{\left \|\xi \right \|}^{ps}_{{\Pi}_{p}^{*,s}(f,h)}+{\left \|\eta  \right \|}^{ps}_{{\Pi}_{p}^{*,s}(f,h)}\right).
		\end{eqnarray}
		Using the relation (\ref{e34}) with $K = {B}^{n}$, we obtain
		\begin{eqnarray*}
			\int_{{\mathbb{R}}^{n}}\int_{{\mathbb{R}}^{n}}\frac{{|f(x)-h(y)|}^{p}}{{|x-y|}^{n+ps}}dxdy=\int_{{\mathbb{S}}^{n-1}}{\left \|\xi\right \|}^{ps}_{{\Pi}_{p}^{*,s}(f,h)}d\xi
		\end{eqnarray*} 
		since  $(f,h)\in{W}^{s,p}({\mathbb{R}}^{n},{\mathbb{R}}^{2})$, the right side of the above equality  is finite. Since $f\in {W}^{s,p}(\mathbb{R}^{n})$, $\int_{{\mathbb{S}}^{n-1}}{\left \|\xi\right \|}^{ps}_{{\Pi}_{p}^{*,s}(f,f)}d\xi$ is bounded as well. We pick $r > 0$ sufficiently large so that the
		set $$D = \{\xi\in{\mathbb{S}}^{n-1}:{\left \|\xi\right \|}^{ps}_{{\Pi}_{p}^{*,s}(f,h)}<r,\;\;{\left \|\xi\right \|}^{ps}_{{\Pi}_{p}^{*,s}(f,f)}<r \}$$ has positive $(n-1)$-dimensional Hausdorff
		measure with a basis $\{{\xi}_{1},\dots,{\xi}_{n}\} \subset D$ of ${\mathbb{R}}^{n}$. If necessary, we perform a linear
		transformation on $f$ and $h$, without loss of generality, making ${\xi}_{i} = {e}_{i}$, the orthonormal basis vectors. For every $x\in{\mathbb{R}}^{n}$, writing $x =\sum {x}_{i}{e}_{i}$ and using (\ref{3c}), we have
		\begin{eqnarray}\label{3e}
			{\left \|x \right \|}_{{\Pi}_{p}^{*,s}(f,h)}
			&\le&\left ( {2}^{p-1 }\displaystyle\sum_{i=1}^{n}{|{x}_{i}|}^{ps}{\left \|{e}_{i}\right \|}^{ps}_{{\Pi}_{p}^{*,s}(f,h) } + {2}^{p-1}\displaystyle\sum_{i=1}^{n}{|{x}_{i}|}^{ps}{\left \|{e}_{i}\right \|}^{ps}_{{\Pi}_{p}^{*,s}(f,f) } \right )^{\frac{1}{ps} } \nonumber\\
			&\le& c_{1}\|x\|,
		\end{eqnarray}
		with $c_{1}> 0$ being independent of $x$, it follows that ${\Pi}_{p}^{*,s}(f,h)$ has the origin as interior point. Similarly, there exists $c_{2}>0$ independent of $x$ such that
		\begin{eqnarray}\label{e3.5}
			{\left \|x \right \|}_{{\Pi}_{p}^{*,s}(f,f)}
			\le c_{2}\|x\|.
		\end{eqnarray}
		Finally, we show that ${||\cdot ||}_{{\Pi}_{p}^{*,s}(f,h)}$ is continuous. For $\xi$, $\eta\in{\mathbb{R}}^{n} \setminus \{0\}$ , by (\ref{3a}), the triangle
		inequality and (\ref{e3.5}), we have
		\begin{eqnarray*}
			&&{\left \|\xi+\eta \right \|}^{ps}_{{\Pi}_{p}^{*,s}(f,h)}\nonumber\\
			&=& \int_{0}^{\infty }{t }^{-ps-1}{\left \|f(\cdot +t\xi+t\eta)-h(\cdot ) \right \|}^{p}_{p}dt\nonumber\\
			&=& \int_{0}^{\infty }{t }^{-ps-1}{\left \|f(\cdot +t\xi+t\eta)-f\left ( \cdot +t\xi  \right )+f\left ( \cdot +t\xi  \right ) -h(\cdot ) \right \|}^{p}_{p}dt\nonumber\\
			&\le&\int_{0}^{\infty }{t }^{-ps-1}\left ( {\left \|f(\cdot +t\xi+t\eta)-f\left ( \cdot +t\xi  \right )\right\|_{p} +\left\|f\left ( \cdot +t\xi  \right ) -h(\cdot ) \right \|_{p} } \right )^{p} dt\nonumber\\
			&\le&{(1+{\|\eta\|}^{\frac{s}{2}\frac{p}{p-1} } )}^{p-1}\int_{0}^{\infty }{t }^{-ps-1}\left ({\|\eta\|}^{-\frac{ps}{2} }{\left \|f\left ( \cdot +t\xi+t\eta  \right )-f\left ( \cdot +t\xi  \right )  \right \|}^{p}_{p}+{\left \|f\left ( \cdot +t\xi  \right )-h\left ( \cdot  \right )   \right \|}^{p}_{p}\right )dt\nonumber\\
			&=&{(1+{\|\eta\|}^{\frac{s}{2}\frac{p}{p-1} })}^{p-1}({\|\eta\|}^{-\frac{ps}{2} }	{\left \|\eta \right \|}^{ps}_{{\Pi}_{p}^{*,s}(f,f)}+{\left \|\xi \right \|}^{ps}_{{\Pi}_{p}^{*,s}(f,h)})\nonumber\\
			&\le& {(1+{\|\eta\|}^{\frac{s}{2}\frac{p}{p-1} })}^{p-1}(c_{2}{\|\eta\|}^{\frac{ps}{2} }+{\left \|\xi \right \|}^{ps}_{{\Pi}_{s}^{*,s}(f,h)}),
		\end{eqnarray*}
		where we used the inequality 
		\begin{eqnarray*}
			a+b\le {(1 + {r}^{p/(p-1)})}^{(p-1)/p}{({({r}^{-1}a)}^{p}  +{b}^{p})}^{1/p}
		\end{eqnarray*} 
		for $a, b, r > 0$, which is a consequence of H\"older's inequality.
		We obtain
		\begin{eqnarray}\label{3f}
			{\left \|\xi+\eta \right \|}^{ps}_{{\Pi}_{p}^{*,s}(f,h)}\le{(1+{\|\eta\|}^{\frac{s}{2}\frac{p}{p-1} })}^{p-1}(c_{2}{\|\eta\|}^{\frac{ps}{2} }+{\left \|\xi \right \|}^{ps}_{{\Pi}_{s}^{*,s}(f,h)}).
		\end{eqnarray}
		Using inequality (\ref{3f}) for the vectors $\xi+\eta$ and $-\eta$, we derive
		\begin{eqnarray*}
			{\left \|\xi \right \|}^{ps}_{{\Pi}_{p}^{*,s}(f,h)}&=&{\left \|\xi+\eta-\eta \right \|}^{ps}_{{\Pi}_{p}^{*,s}(f,h)}\nonumber\\
			&\le&{(1+{\|\eta\|}^{\frac{s}{2}\frac{p}{p-1} })}^{p-1}(c_{2}{\|-\eta\|}^{\frac{ps}{2} }+{\left \|\xi+\eta \right \|}^{ps}_{{\Pi}_{s}^{*,s}(f,h)})
		\end{eqnarray*} 
		which implies
		\begin{eqnarray}\label{3g}
			{\left \|\xi+\eta \right \|}^{ps}_{{\Pi}_{p}^{*,s}(f,h)}\ge{(1+{\|\eta\|}^{\frac{s}{2}\frac{p}{p-1} })}^{1-p} {\left \|\xi \right \|}^{ps}_{{\Pi}_{p}^{*,s}(f,h)}-c_{2}{\|\eta\|}^{\frac{ps}{2} }.
		\end{eqnarray}
		  From (\ref{3f}) and (\ref{3g}), we deduce the continuity of ${||\cdot ||}_{{\Pi}_{p}^{*,s}(f,h)}$.
	\end{proof}
	
	\section{Fractional asymmetric $L_p$ polar projection bodies}\label{s4}
	Let $0<s<1$ and $1< p< \frac{n}{s}$. For  measurable functions $f,h:\mathbb{R}^{n} \to \mathbb{R}$, define the {\it generalized asymmetric $s$-fractional $L_{p} $ polar projection bodies} $ {\textstyle \Pi_{p,+}^{*,s}}\left ( f,h \right )$ and $ {\textstyle \Pi_{p,-}^{*,s}}\left ( f,h \right )$ as the star-shaped sets given by the gauge functions
	\begin{eqnarray}\label{gf}
		\left \| \xi  \right \|_{{\textstyle \Pi_{p,\pm }^{*,s}}\left ( f,h \right )
		} ^{ps}=\int_{0}^{\infty }t^{-ps-1}\int_{\mathbb{R}^{n} }\left ( f\left ( x+t\xi  \right )-h\left ( x \right )   \right )_{\pm } ^{p}dxdt    
	\end{eqnarray}
	for $\xi \in \mathbb{R}^{n}$. We have ${\textstyle \Pi_{p,-}^{*,s}}\left ( f,h \right )={\textstyle \Pi_{p,+}^{*,s}}\left ( -f,-h \right )=-{\textstyle \Pi_{p,+}^{*,s}}\left ( h,f \right )$ and sate our results just for ${\textstyle \Pi_{p,+}^{*,s}}\left ( f,h \right )$. We remark that $
	\left \| \cdot   \right \|_{{\textstyle \Pi_{p,\pm }^{*,s}}\left ( f,h \right )
	} ^{ps}$ is a one-homogeneous function on $\mathbb{R}^{n} $, analogous to the symmetric case. Also, note that
	\begin{eqnarray}\label{4.1}
		\left \| \xi  \right \|_{{\textstyle \Pi_{p }^{*,s}}\left ( f,h \right )} ^{ps}=\left \| \xi  \right \|_{{\textstyle \Pi_{p,+ }^{*,s}}\left ( f,h \right )} ^{ps}+\left \| \xi  \right \|_{{\textstyle \Pi_{p,- }^{*,s}}\left ( f,h \right )} ^{ps}
	\end{eqnarray}
	for $\xi \in \mathbb{R}^{n}$.
	
	Let $K\subset \mathbb{R}^{n}$ be a star body and $\left ( f,h \right )\in W^{s,p}\left ( \mathbb{R}^{n},\mathbb{R}^{2}   \right ) $. As in (\ref{3b}), we obtain that
	\begin{eqnarray}\label{4.b}
		\int_{\mathbb{R}^{n} }\int_{\mathbb{R}^{n} }\frac{\left ( f\left ( x \right )-h\left ( y \right )   \right )_{+} ^{p}  }{\left \| x-y \right \|_{K} ^{n+ps}  }dxdy=n\widetilde{V }_{-ps}\left ( K,{\textstyle \Pi_{p,+}^{*,s}}\left ( f,h \right ) \right ). 
	\end{eqnarray}
	
	In the following proposition, we prove the basic properties of generalized $s$-fractional  $L_{p}$ polar projection bodies with asymmetry.
	\begin{proposition}
		For $p>1$ and non-zero $\left ( f,h \right )\in W^{s,p}\left ( \mathbb{R}^{n},\mathbb{R}^{2}\right )$ and $f\in {W}^{s,p}(\mathbb{R}^{n})$, the set ${\textstyle \Pi_{p,+}^{*,s}}\left ( f,h \right )$ is a star body with the origin in its interior. Moreover, there is $c>0$ depending only on $f$, $h$ and $p$ such that ${\textstyle \Pi_{p,+}^{*,s}}\left ( f,h \right )\subseteq cB^{n}$ for every $s\in \left ( 0,1 \right )$.
	\end{proposition}
	\begin{proof}
		By the convexity of  the functions $\left ( a \right )^{p} _{+}$ and $\left ( a \right )^{p} _{-} $, the inequalities $\left ( a+b \right )^{p} _{+}\ge \left ( a \right )^{p} _{+}+p\left ( a \right )^{p-1} _{+}b  $ and $\left ( a+b \right )^{p} _{-}\ge \left ( a \right )^{p} _{-}-p\left ( a \right )^{p-1} _{-}b$ hold for $a,b\in \mathbb{R}$.
		
		If $\int_{\mathbb{R}^{n} }\left ( f\left ( x \right )   \right )_{+} ^{p}dx>0$, take $\varepsilon >0$ sufficiently small such that $\varepsilon +p\varepsilon ^{\frac{1}{p} }\left \| f \right \|^{p-1} _{p}\le \frac{1}{2}\int_{\mathbb{R}^{n} }\left ( f\left ( x \right )  \right )_{+} ^{p}dx$, and choose $r>0$ large enough so that  $\int_{\mathbb{R}^{n}\setminus r B^{n}  }\left (f\left ( x \right )  \right )^{p}_{+}dx<\varepsilon $ and $\int_{\mathbb{R}^{n}\setminus r B^{n}  }\left | h\left ( x \right )  \right |^{p}dx<\varepsilon $. Then for $z\in \mathbb{R}^{n}\setminus 2rB^{n} $, we obtain by H\"{o}lder's inequality that
		\begin{eqnarray}
			&&\int_{rB^{n} }\left ( f\left ( x \right )-h\left ( x+z \right )   \right )_{+} ^{p}dx\nonumber\\
			&\ge&\int_{rB^{n} }\left ( f\left ( x \right )  \right )_{+} ^{p}-p\left ( f\left ( x \right )  \right )_{+} ^{p-1}h\left ( x+z \right )dx\nonumber\\
			&\ge&\int_{rB^{n} }\left ( f\left ( x \right )  \right )_{+} ^{p}dx-p\left ( \int_{rB^{n} }\left ( f\left ( x \right )  \right )_{+} ^{p}dx    \right )^{\frac{p-1}{p} }\left ( \int_{rB^{n} }\left | h\left ( x+z \right )  \right |^{p}dx    \right )^{\frac{1}{p} }\nonumber\\
			&\ge&\int_{rB^{n} }\left ( f\left ( x \right )  \right )_{+} ^{p}dx-p\left ( \int_{\mathbb{R}^{n} }\left ( f\left ( x \right )  \right )_{+} ^{p}dx    \right )^{\frac{p-1}{p} }\left ( \int_{\mathbb{R}^{n}\setminus rB^{n}  }\left | h\left ( x \right )  \right |^{p}dx    \right )^{\frac{1}{p} }\nonumber\\
			&\ge& \int_{\mathbb{R}^{n} }\left ( f\left ( x \right )  \right )^{p} _{+}dx-\varepsilon -p\left \| f \right \|_{p} ^{p-1}\varepsilon ^{\frac{1}{p} } \nonumber\\
			&\ge&\frac{1}{2}\int_{\mathbb{R}^{n} }\left ( f\left ( x \right )  \right )_{+} ^{p}dx \nonumber.
		\end{eqnarray} 
		When $\int_{\mathbb{R}^{n} }\left ( f\left ( x \right )  \right )^{p} _{+}dx=0$, the previous inequality is trivially true for any  $r>0$.
		
		If $\int_{\mathbb{R}^{n} }\left ( h\left ( x \right )   \right )_{-} ^{p}dx>0$, take $\varepsilon >0$ sufficiently small such that $\varepsilon +p\varepsilon ^{\frac{1}{p} }\left \| h \right \|^{p-1} _{p}\le \frac{1}{2}\int_{\mathbb{R}^{n} }\left ( h\left ( x \right )  \right )_{-} ^{p}dx$, and choose $r>0$  large enough so that $\int_{\mathbb{R}^{n}\setminus r B^{n}  }\left ( h\left ( x \right )  \right )^{p}_{-}dx<\varepsilon $ and  $\int_{\mathbb{R}^{n}\setminus r B^{n}  }\left | f\left ( x \right )  \right |^{p}dx<\varepsilon $. Then for $z\in \mathbb{R}^{n}\setminus 2rB^{n} $, we obtain by H\"{o}lder's inequality that
		\begin{eqnarray}
			&&\int_{rB^{n}-z }\left ( f\left ( x \right )-h\left ( x+z \right )   \right )^{p} _{+}dx\nonumber\\
			&=&\int_{rB^{n} }\left ( h\left ( x \right )-f\left ( x-z \right )   \right )^{p} _{-}dx\nonumber\\
			&\ge&\int_{rB^{n} }\left (h\left  ( x \right )  \right )_{-} ^{p}+p\left ( h\left ( x \right )  \right )_{-} ^{p-1}f\left ( x-z \right )dx\nonumber\\
			&\ge&\int_{rB^{n} }\left ( h\left ( x \right )  \right )_{-} ^{p}-p\left ( h\left ( x \right )  \right )_{-} ^{p-1}\left|f\left ( x-z \right )\right|dx\nonumber\\
			&\ge&\int_{rB^{n} }\left ( h\left ( x \right )  \right )_{-} ^{p}dx-p\left ( \int_{rB^{n} }\left ( h\left ( x \right )  \right )_{-} ^{p}dx    \right )^{\frac{p-1}{p} }\left ( \int_{rB^{n} }\left | f\left ( x-z \right )  \right |^{p}dx    \right )^{\frac{1}{p} }\nonumber\\
				&\ge&\int_{rB^{n} }\left ( h\left ( x \right )  \right )_{-} ^{p}dx-p\left ( \int_{\mathbb{R}^{n} }\left ( h\left ( x \right )  \right )_{-} ^{p}dx    \right )^{\frac{p-1}{p} }\left ( \int_{\mathbb{R}^{n}\setminus rB^{n}  }\left | f\left ( x \right )  \right |^{p}dx    \right )^{\frac{1}{p} }\nonumber\\
			&\ge& \int_{\mathbb{R}^{n} }\left ( h\left ( x \right )  \right )^{p} _{-}dx-\varepsilon -p\left \| h \right \|_{p} ^{p-1}\varepsilon ^{\frac{1}{p} } \nonumber\\
			&\ge&\frac{1}{2}\int_{\mathbb{R}^{n} }\left ( h\left ( x \right )  \right )^{p} _{-}dx.\nonumber
		\end{eqnarray}
		In case $\int_{\mathbb{R}^{n} }\left ( h\left ( x \right )   \right )_{-} ^{p}dx=0$ the inequality is satisfied trivially for arbitary $r>0$.
		
		It follows that $\int_{\mathbb{R}^{n} }\left ( f\left ( x \right )-h\left ( x+z \right )   \right )^{p} _{+}dx\ge  \frac{1}{2}\int_{\mathbb{R}^{n} }\left ( f\left ( x \right )  \right )_{+} ^{p}dx+\frac{1}{2}\int_{\mathbb{R}^{n} }\left ( h\left ( x \right )  \right )_{-} ^{p}dx$ for every $z\in \mathbb{R}^{n}\setminus 2rB^{n}$ with $r>0$ depending only on $f$, $h$. Finally, 
		\begin{eqnarray}
			\left \| \xi  \right \|_{{\textstyle \Pi_{p,+ }^{*,s}}\left ( f,h \right )} ^{ps} &\ge& \int_{2r}^{\infty  }t^{-ps-1}\int_{\mathbb{R}^{n} }\left ( f\left ( x \right )-h\left ( x-t\xi  \right )   \right )_{+ } ^{p}dxdt\nonumber\\
			&\ge&\int_{2r}^{\infty }t^{-ps-1}dt \frac{1}{2}\int_{\mathbb{R}^{n} }\left ( f\left ( x \right )  \right )_{+}^{p}+\left ( h\left ( x \right )  \right )_{-} ^{p}dx\nonumber\\
			&\ge&\frac{\left ( 2r \right )^{-ps}  }{ps}\frac{1}{2}\int_{\mathbb{R}^{n} }\left (f\left ( x \right )  \right )_{+}^{p}+\left ( h\left ( x \right )  \right )_{-} ^{p}dx\nonumber\\
			&\ge& c^{ps}\nonumber, 
		\end{eqnarray}
	which implies that ${\Pi}_{p,+}^{*,s}(f,h)\subset c{B}^{n}$ for $c>0$ independent of $s$.
	
		Note that $ {\textstyle \Pi_{p}^{*,s}}\left ( f,h \right )\subset{\textstyle \Pi_{p,+}^{*,s}}\left ( f,h \right )$. Therefore, it follows frome Proposition \ref{3.a} that ${\textstyle \Pi_{p,+}^{*,s}}\left ( f,h \right ) $ contains the origin in its interior, that is, there is $c_{3}>0$ such that 
		\begin{eqnarray}\label{4.c}
			\left \| x \right \|_{{\textstyle \Pi_{p,+}^{*,s}}\left ( f,h \right ) }\le c_{3}\left \| x \right \|\label{1.3}
		\end{eqnarray}
		for $x\in \mathbb{R}^{n} $. Similarly, there exists $c_{4}>0$ independent of $x$ such that
		\begin{eqnarray}\label{e4.5}
		\left \| x \right \|_{{\textstyle \Pi_{p,+}^{*,s}}\left ( f,f \right ) }\le c_{4}\left \| x \right \|.
		\end{eqnarray}
		
		Finally, we show that $\left \| \cdot   \right \|_{{\textstyle \Pi_{p,+ }^{*,s}}\left ( f,h \right )} ^{ps}$ is continuous. Observe that the inequality $\left ( a+b \right )^{p} _{+}\le \left ( a_{+}+b_{+}   \right )^{p} $ holds for any $a,b\in \mathbb{R}$. Hence, for $\xi ,\eta\in \mathbb{R}^{n} \setminus\{0\} $, we obtain that
		\begin{eqnarray}
			&&\int_{\mathbb{R}^{n} }\left ( f\left ( x+t\xi +t\eta  \right )-h\left ( x \right )   \right )^{p} _{+}dx\nonumber\\
			&=&\int_{\mathbb{R}^{n} }\left (f\left ( x+t\xi +t\eta  \right )-f\left ( x+t\xi   \right )+f\left ( x+t\xi   \right )-h\left ( x \right ) \right )^{p} _{+}dx\nonumber\\
			&\le&\int_{\mathbb{R}^{n} }\left (\left( f\left ( x+t\xi +t\eta  \right )-f\left ( x+t\xi   \right )\right)_{+} +\left(f\left ( x+t\xi   \right )-h\left ( x \right )\right)_{+} \right )^{p} dx\nonumber\\
			&\le& \int_{\mathbb{R}^{n} }\left ( 1+\left \| \eta  \right \|^{\frac{s}{2}\cdot \frac{p}{p-1}  }   \right )^{p-1}\left ( \left \| \eta  \right \|^{-\frac{ps}{2}}\left( f\left ( x+t\xi +t\eta  \right )-f\left ( x+t\xi   \right )\right)^{p} _{+}  +\left(f\left ( x+t\xi   \right )-h\left ( x \right )\right)^{p} _{+}   \right )dx\nonumber   
		\end{eqnarray}
		where we used the inequality 
		\begin{eqnarray*}
			a+b\le {(1 + {r}^{p/(p-1)})}^{(p-1)/p}{({({r}^{-1}a)}^{p}  +{b}^{p})}^{1/p}
		\end{eqnarray*} 		
		for $a, b, r > 0$, which is a consequence of H\"older’s inequality. Thus, integrating and using (\ref{e4.5}), for $c_{4}>0$,  we obtain
		\begin{eqnarray}
			\left \| \xi +\eta  \right \|_{ {\textstyle \Pi_{p,+}^{\ast ,s}}\left ( f,h \right )} ^{ps} &\le&\left ( 1+\left \| \eta  \right \|^{\frac{s}{2}\cdot \frac{p}{p-1}  }   \right )^{p-1}\left ( \left \| \eta  \right \|^{-\frac{ps}{2} }\left \| \eta  \right \|_{ {\textstyle \Pi_{p,+}^{\ast ,s}}\left ( f,f \right )  } ^{ps}+ \left \| \xi  \right \|_{ {\textstyle \Pi_{p,+}^{\ast ,s}}\left ( f,h \right )  } ^{ps}   \right )\nonumber\\
			&\le&\left ( 1+\left \| \eta  \right \|^{\frac{s}{2}\cdot \frac{p}{p-1}  }   \right )^{p-1}\left ( c_{4}\left \| \eta  \right \|^{\frac{ps}{2} }+\left \| \xi   \right \|_{ {\textstyle \Pi_{p,+}^{\ast ,s}}\left ( f,h \right )  } ^{ps}\right ).\label{1.4}
		\end{eqnarray}
		Applying inequality \eqref{1.4} to the vectors $\xi +\eta $ and $-\eta$, we have
		\begin{eqnarray}
			\left \| \xi   \right \|_{ {\textstyle \Pi_{p,+}^{\ast ,s}}\left ( f,h \right )  } ^{ps}&=&\left \| \xi +\eta-\eta   \right \|_{ {\textstyle \Pi_{p,+}^{\ast ,s}}\left ( f,h \right )  } ^{ps}\nonumber\\
			&\le& \left ( 1+\left \| -\eta  \right \|^{\frac{s}{2}\cdot \frac{p}{p-1}  }   \right )^{p-1}\left ( c_{4}  \left \| -\eta  \right \|^{\frac{ps}{2} }+\left \| \xi +\eta \right \|_{ {\textstyle \Pi_{p,+}^{\ast ,s}}\left ( f,h \right )  } ^{ps}\right),\nonumber
		\end{eqnarray}
		which implies
		\begin{eqnarray}
			\left \| \xi +\eta \right \|_{ {\textstyle \Pi_{p,+}^{\ast ,s}}\left ( f,h \right )  } ^{ps}\ge \left ( 1+\left \| \eta  \right \|^{\frac{s}{2}\cdot \frac{p}{p-1}  }   \right )^{1-p}\cdot \left \| \xi   \right \|_{ {\textstyle \Pi_{p,+}^{\ast ,s}}\left ( f,h \right )  } ^{ps}-c_{4}  \left \| \eta  \right \|^{\frac{ps}{2} } \label{1.5}.
		\end{eqnarray}
		The continuity of $\left \| \cdot   \right \|_{{\textstyle \Pi_{p,+ }^{*,s}}\left ( f,h \right )} ^{ps}$ can be derived from \eqref{1.4} and \eqref{1.5}.
	\end{proof}
	\section{Generalized anisotropic fractional $L_p$ P\'{o}lya–Szeg\"{o} inequalities}\label{s5}
	We present new anisotropic P\'{o}lya–Szeg\"{o} inequalities for fractional $L_{p}$ Sobolev norms, including both symmetric and asymmetric settings.
	\begin{lemma}\label{5.a} Let $K \subset \mathbb{R}^n$ be a star body, if $f \in L^p(\mathbb{R}^n)$, $h \in L^p(\mathbb{R}^n)$ are non-negative,  then
		\begin{eqnarray}\label{5a}
			\int_{\mathbb{R}^n}\int_{\mathbb{R}^n}\frac{{(f\left ( x \right )-h\left ( y \right ))} _{+} ^{p}}{\left \| x-y \right \|^{n+ps} _{K} }dx dy
			\ge \int_{\mathbb{R}^n}\int_{\mathbb{R}^n}\frac{{(f^{\star}\left ( x \right )-h^{\star}\left ( y \right ))} _{+} ^{p}}{\left \| x-y \right \|^{n+ps} _{K^{\star}} }dxdy.   
		\end{eqnarray}
		Equality holds for non-zero  $\left ( f,h \right )\in W^{s,p}\left ( \mathbb{R}^{n},\mathbb{R}^{2}   \right )$ if and only if $K$ is an ellipsoid,
		 $f= f^{\star }(\phi x+x_{0} )$, $h= h^{\star }(\phi x+x_{0} )$ for some $\phi \in GL(n)$ and $x_{0}\in \mathbb{R}^n$.
	\end{lemma}
	\begin{proof}
		Writing
		\begin{equation*}
			{||z||}^{-n-ps}_{K}=\int_{0}^{\infty }k_{t}(z)dt,  
		\end{equation*}
		where $k_{t}(z)=1_{t^{-1/(n+ps)}K}(z)$, we obtain 
		\begin{equation*}
			\int_{\mathbb{R}^n}\int_{\mathbb{R}^n}\frac{{(f\left ( x \right )-h\left ( y \right ))} _{+} ^{p}}{\left \| x-y \right \|^{n+ps} _{K} }dxdy=\int_{0}^{\infty }\int_{\mathbb{R}^n}\int_{\mathbb{R}^n}{(f\left ( x \right )-h\left ( y \right ))} _{+} ^{p}
			k_{t}(x-y)dxdydt.
		\end{equation*}
		Note that
		\begin{equation}\label{e52}
			{(f\left ( x \right )-h\left ( y \right ))} _{+} ^{p}=p\int_{0}^{\infty}{(f\left ( x \right )-r)} _{+} ^{p-1}1_{\{h<r\}}(y)dr.
		\end{equation}
		 By (\ref{e52}), when $t>0$, we have 
		\begin{eqnarray*}
			&&\int_{\mathbb{R}^n}\int_{\mathbb{R}^n}{(f\left ( x \right )-h\left ( y \right ))} _{+}^{p}k_{t}(x-y)dxdy\nonumber\\
			&=&p\int_{0}^{\infty}\int_{\mathbb{R}^n}\int_{\mathbb{R}^n}{(f\left ( x \right )-r)} _{+} ^{p-1}
			k_{t}(x-y)1_{\{h<r\}}(y)dxdydr\nonumber\\ 
			&=&p\int_{0}^{\infty}\int_{\mathbb{R}^n}\int_{\mathbb{R}^n}{(f\left ( x \right )-r)} _{+} ^{p-1}
			k_{t}(x-y)(1-1_{\{h\ge r\}}(y))dxdydr.  
		\end{eqnarray*}

And note that
		\begin{eqnarray}
			&&p\int_{\mathbb{R}^n}\int_{\mathbb{R}^n}{(f\left ( x \right )-r)} _{+} ^{p-1}
			k_{t}(x-y)(1-1_{\{h\ge r\}}(y))dxdy\nonumber\\
			&=&p{||k_{t}||}_{1}\int_{\mathbb{R}^n}{(f\left ( x \right )-r)} _{+} ^{p-1}dx-p\int_{\mathbb{R}^n}\int_{\mathbb{R}^n}{(f\left ( x \right )-r)} _{+} ^{p-1}
			k_{t}(x-y)1_{\{h\ge r\}}(y)dxdy.\nonumber\\\label{e78}
		\end{eqnarray}
		
		Next, we  prove $\int_{\mathbb{R}^n}{(f\left ( x \right )-r)} _{+} ^{p-1}dx<\infty$. By  H\"{o}lder's inequality, we obtain
		\begin{eqnarray*}
			\int_{\mathbb{R}^n}{(f\left ( x \right )-r)} _{+} ^{p-1}dx
			=\int_{\left \{ f(x)\ge r \right \} }{(f\left ( x \right )-r)}^{p-1}dx
			\le\left (  \int_{\left \{ f(x)\ge r \right \}} \left ( f(x)-r \right )^{p}dx\right )^{\frac{p-1}{p}}\left |\left \{  f(x)\ge r \right \}  \right |^{\frac{1}{p} }  .
		\end{eqnarray*}
		Then we only need to prove $\int_{\left \{ f(x)\ge r \right \}}{(f\left ( x \right )-r)} ^{p}dx<\infty$. Let $g(x)=f(x)-r $ for every $x\in \mathbb{R}^{n}$, for $f \in L^p(\mathbb{R}^n)$, by Fubini's theorem, we have
		\begin{eqnarray*}
			&&\int_{\left \{ f(x)\ge r \right \}}{(f\left ( x \right )-r)}^{p}dx=\int_{\left \{ g(x)>0 \right \} } g(x)^{p}dx=\int_{0}^{\infty }\left | \left \{ {g\left ( x \right )^{p}>t} \right \} \right |dt\nonumber\\
			& =&\int_{0}^{\infty }\left | \left \{ {g\left ( x \right )>t^{1/p}} \right \} \right |dt =\int_{0}^{\infty }\left | \left \{ f\left ( x \right )>r+t^{1/p} \right \} \right |
			\le\int_{0}^{\infty }\left | \left \{ f\left ( x \right )>t^{1/p} \right \} \right |dt \nonumber\\
			&=&\int_{0}^{\infty }\left | \left \{ f\left ( x \right )^{p}>t \right \} \right |dt
			=\int_{\left \{ f(x)>0 \right \} }\left | f(x) \right |^{p} dx< \infty.       
		\end{eqnarray*}
		 We have proved that the first term of (\ref{e78}) is finite. Clearly, the first term is invariant under symmetric decreasing rearrangement. For the second term of (\ref{e78}), by the Riesz rearrangement inequality, i.e., Theorem \ref{Rri}, we have
		\begin{eqnarray*}
			&&\int_{\mathbb{R}^n}\int_{\mathbb{R}^n}{(f\left ( x \right )-r)} _{+} ^{p-1}
			k_{t}(x-y)1_{\{h\ge r\}}(y)dxdy\nonumber\\ 
			&\leq& \int_{\mathbb{R}^n}\int_{\mathbb{R}^n}{(f^{\star}\left ( x \right )-r)} _{+} ^{p-1}
			k^{\star}_{t}(x-y)1_{\{h^{\star}\ge r\}}(y)dxdy
		\end{eqnarray*}
		for $r$, $t>0$. Note that
		\begin{equation*}
			{(f\left ( x \right )-r)} _{+} ^{p-1}=(p-1)\int_{0}^{\infty}{(\tilde{r}-r)} _{+} ^{p-2}1_{\{f\geq \tilde{r}\}}(x)d\tilde{r},
		\end{equation*}
		and that the corresponding equation holds for $f^{\star}$. Hence, if there is equality in (\ref{5a}),
		then, for $(\tilde{r},\;r,\;t)\in(0,\infty)^3\setminus M$ with $M\subset(0,\infty)^3$ and $|M|=0$, we have
		\begin{eqnarray*}
			&&\int_{\mathbb{R}^n}\int_{\mathbb{R}^n}1_{\{f\geq \tilde{r}\}}(x)1_{t^{-1/(n+ps)K}}(x-y)1_{\{h\geq r\}}(y)dxdy\nonumber\\ 
			&=&\int_{\mathbb{R}^n}\int_{\mathbb{R}^n}1_{\{f^{\star}\geq \tilde{r}\}}(x)1_{t^{-1/(n+ps)K^{\star}}}(x-y)1_{\{h^{\star}\geq r\}}(y)dxdy.\nonumber
		\end{eqnarray*}
		For almost every $(\tilde{r},\;r)\in(0,\infty)^2$, we have $(\tilde{r},\;r,\;t)\in(0,\infty)^3\setminus M$ for almost every
		$t>0$. Let $\alpha$, $\beta$ and $\gamma$  be the radii of ${\{f\geq \tilde{r}\}}^{\star}$, $t^{-1/(n+ps)}K^{\star}$ and ${\{h\geq r\}}^{\star}$, respectively. 
		Fixed $\tilde{r}$ and $ t$, 
		then there exists a closed interval $[{\bar{c}}_{\tilde{r},t},{\tilde{c}}_{\tilde{r},t}]$ with respect to $ \tilde{r} $ and $ t $ such that $ |\alpha-\beta| <\gamma<\alpha+\beta$ for any $r\in[{\bar{c}}_{\tilde{r},t},{\tilde{c}}_{\tilde{r},t}]$. Thus by Theorem \ref{2.b}, there exists an ellipsoid $D$ and vectors $a, b \in \mathbb{R}^n$ such that  
		\begin{eqnarray}
			\{f\geq \tilde{r}\}=a+\alpha D,\;t^{-1/(n+ps)}K=b+\beta D,\;\{h\geq r\}={c}_{r}+\gamma D
			\notag
		\end{eqnarray}  
		where ${c}_{r} = a + b$. Then for $r\in[{\bar{c}}_{\tilde{r},t},{\tilde{c}}_{\tilde{r},t}]$, ${c}_{r}$ is a constant vector. When $\tilde{r}$ and $t$ vary, the interval $\displaystyle\bigcup_{\tilde{r},t>0} [{\bar{c}}_{\tilde{r},t},{\tilde{c}}_{\tilde{r},t}]$ covers $(0, +\infty)$, thus for any $r\in(0, +\infty)$, ${c}_{r}$ is a constant vector.
		Since $K={t}^{1/(n+ps)}b+{(|K|/  
			|D|)}^{1/n}D$, the ellipsoid $D$
		is independent of $(\tilde{r},\;r,\;t)$, , which implies that the vector $b$ is a constant vector. Thus vector $a$ is also a constant vector. This completes the proof.
	\end{proof}
	Next, we prove the symmetric version of new anisotropic P\'{o}lya–Szeg\"{o} inequalities.     
	\begin{lemma}\label{1.2}
		 Let $K \subset \mathbb{R}^n$ be a star body, if $f \in L^p(\mathbb{R}^n)$, $h \in L^p(\mathbb{R}^n)$ are non-negative,  then
		\begin{eqnarray}\label{e5.3}
			\int_{\mathbb{R}^{n} }\int_{\mathbb{R}^{n} }\frac{\left | f\left ( x \right )-h\left ( y \right ) \right | ^{p}    }{\left \| x-y \right \|^{n+ps} _{K}  }dxdy\ge\int_{\mathbb{R}^{n} }\int_{\mathbb{R}^{n} }\frac{\left | f^{\star } \left ( x \right )-h^{\star } \left ( y \right ) \right | ^{p}   }{\left \| x-y \right \|^{n+ps} _{K^{\star } }  }dxdy. 
		\end{eqnarray}
		Equality holds for non-zero $\left ( f,h \right )\in W ^{s,p} \left ( \mathbb{R}^{n},\mathbb{R}^{2}   \right )$ if and only if $K$ is an ellipsoid, $f= f^{\star }(\phi x+x_{0} )$, $h= h^{\star }(\phi x+x_{0} )$ for some $\phi \in GL(n)$ and $x_{0}\in \mathbb{R}^n$.. \nonumber
	\end{lemma}
	\begin{proof}
	By Lemma \ref{5.a}, we obtain 
	\begin{eqnarray}
	\int_{\mathbb{R}^{n} }\int_{\mathbb{R}^{n} }\frac{\left ( f\left ( x \right )-h\left ( y \right )  \right )^{p} _{-}}{\left \| x-y \right \|^{n+ps} _{K} }dxdy&=&\int_{\mathbb{R}^{n} }\int_{\mathbb{R}^{n} }\frac{\left ( h\left ( y \right )-f\left ( x \right )  \right )^{p} _{+}}{\left \| y-x \right \|^{n+ps} _{-K} }dxdy\nonumber\\
	&\ge&\int_{\mathbb{R}^{n} }\int_{\mathbb{R}^{n} }\frac{\left ( h^{\star } \left ( y \right )-f^{\star } \left ( x \right )  \right )^{p} _{+}}{\left \| y-x \right \|^{n+ps} _{(-K)^{\star } } }dxdy\nonumber\\
	&=&\int_{\mathbb{R}^{n} }\int_{\mathbb{R}^{n} }\frac{\left ( f^{\star } \left ( x \right )-h^{\star } \left ( y \right )  \right )^{p} _{-}}{\left \| x-y \right \|^{n+ps} _{K^{\star } } }dxdy\label{e5.4}.
	\end{eqnarray}
	By (\ref{5a}), (\ref{e5.4}) and $\left | a \right |=a_{+} +a_{-} $ for $a\in \mathbb{R}   $, we can get the desired inequality (\ref{e5.3}). The equality case in inequality in (\ref{e5.3}) follows from  the equality case of Lemma \ref{5.a}.
	\end{proof}
	\section{ Affine fractional $L_{p}$ P\'{o}lya-Szeg\"{o} inequalities}\label{s6}
	We  establish  affine fractional $L_{p}$ P\'{o}lya-Szeg\"{o} inequalities for genrealized fractional asymmetric and symmetric $L_{p}$ polar projection bodies.
	\begin{lemma}\label{1.3}
		If $\left ( f,h \right )\in W^{s,p}\left ( \mathbb{R}^{n},\mathbb{R}^{2}   \right )$ and $f\in W^{s,p}(\mathbb{R}^{n})$ are non-negative, then
		\begin{eqnarray}\label{6.1}
			\left |  {\textstyle\Pi_{p,+}^{*,s}}\left ( f,h \right )  \right |^{-ps/n}\ge \left | {\textstyle \Pi_{p,+}^{*,s}}\left ( f^{\star } ,h^{\star }  \right )  \right |^{-ps/n}.
		\end{eqnarray}
		Equality holds if and only if  $f= f^{\star }(\phi x+x_{0} )$, $h= h^{\star }(\phi x+x_{0} )$ for some $\phi \in GL(n)$ and $x_{0}\in \mathbb{R}^n$.
	\end{lemma}
	\begin{proof}
		By Lemma \ref{5.a}, (\ref{4.b}) and the dual mixed volume inequality (\ref{2a}), we obtain for $K\subset \mathbb{R}^{n}$ a star body that
		\begin{eqnarray}\label{6.2}
			\widetilde{V}_{-ps}\left ( K, {\textstyle\Pi_{p,+}^{*,s}}\left ( f,h \right )\right )&\ge&\widetilde{V}_{-ps}\left ( K^{\star } , {\textstyle\Pi_{p,+}^{*,s}}\left ( f^{\star } ,h^{\star }  \right )\right )\nonumber\\
			&\ge& \left |K^{\star }   \right |^{\left ( n+ps \right )/n }\left |  { {\textstyle \Pi_{p,+}^{\ast,s }} \left ( f^{\star },h^{\star}   \right )}  \right|^{-ps/n}\nonumber\\
			&=&  \left |K   \right |^{\left ( n+ps \right )/n }\left |  { {\textstyle \Pi_{p,+}^{\ast,s }} \left ( f^{\star },h^{\star}   \right )}  \right|^{-ps/n}.  
		\end{eqnarray}
		Seting $K= {\textstyle \Pi_{p,+}^{\ast,s  }}\left ( f,h \right )  $ in (\ref{6.2}), we see that
		\begin{eqnarray}
			\left | {\textstyle \Pi_{p,+}^{\ast,s  }}\left ( f,h \right )  \right |
			\ge \left | {\textstyle \Pi_{p,+}^{\ast,s  }}\left ( f,h \right )  \right |^{\left ( n+ps \right )/n }\left | {\textstyle \Pi_{p,+}^{\ast,s  }}\left ( f^{\star } ,h^{\star }  \right )  \right |^{-ps /n },
		\end{eqnarray}
		which implies  the inequality (\ref{6.1}). By Lemma \ref{5.a}, there is equality in (\ref{6.1}) if and only if  $f= f^{\star }(\phi x+x_{0} )$, $h= h^{\star }(\phi x+x_{0} )$ for some $\phi \in GL(n)$ and $x_{0}\in \mathbb{R}^n$.
	\end{proof}
	The following result can be obtained in the same way as Lemma \ref{1.3} by using  Lemma \ref{1.2} and  (\ref{3b}) instead of Lemma \ref{5.a} and (\ref{4.b}).
	\begin{lemma}\label{e6.2}
		If $f,h\in L^{p} \left ( \mathbb{R}^{n}   \right ) $ are non-negative, then
		\begin{eqnarray}
			\left |  {\textstyle\Pi_{p}^{*,s}}\left ( f,h \right )  \right |^{-ps/n}\ge \left | {\textstyle \Pi_{p}^{*,s}}\left ( f^{\star } ,h^{\star }  \right )  \right |^{-ps/n}.
		\end{eqnarray}
		Equality holds for $\left ( f,h \right )\in W^{s,p}\left ( \mathbb{R}^{n},\mathbb{R}^{2}   \right )$, $f\in W^{s,p}(\mathbb{R}^{n})$  if and only if $f= f^{\star }(\phi x+x_{0} )$, $h= h^{\star }(\phi x+x_{0} )$ for some $\phi \in GL(n)$ and $x_{0}\in \mathbb{R}^n$.
	\end{lemma}
	
	We establish the following asymmetric affine fractional  $L_p$ P\'{o}lya-Szeg\"{o} inequalities.
	\begin{theorem}\label{Th7.1} Let $0 < s < 1$ and $1 < p < n/s$. For non-negative $(f,h)\in{W}^{s,p}(\mathbb{R}^{n},\mathbb{R}^{2})$  and $f\in W^{s,p}(\mathbb{R}^{n})$,
		\begin{eqnarray*}
			\int_{\mathbb{R}^n}\int_{\mathbb{R}^n}\frac{{(f\left ( x \right )-h\left ( y \right ))} _{+} ^{p}}{{|x-y|}^{n+ps}  }dxdy\ge n{\omega} _{n}^{\frac{n+ps}{n}} |{\Pi }_{p,+}^{*,s}(f,h)|^{-\frac{ps}{n}}\ge \int_{\mathbb{R}^n}\int_{\mathbb{R}^n}\frac{{(f^{\star}\left ( x \right )-h^{\star}\left ( y \right ))} _{+} ^{p}}{{|x-y|}^{n+ps}}dxdy.   
		\end{eqnarray*}
		 There is equality in the first inequality if $f$, $h$ are radially symmetric.There is equality in the second inequality if and only if $f= f^{\star }(\phi x+x_{0} )$, $h= h^{\star }(\phi x+x_{0} )$ for some $\phi \in GL(n)$ and $x_{0}\in \mathbb{R}^n$.
	\end{theorem}
	\begin{proof}
	For the first inequality, we set $K = B^n$ in (\ref{4.b}) and apply the dual mixed volume
inequality (\ref{2a}) to obtain
\begin{eqnarray}\label{7.3}
	\int_{\mathbb{R}^n}\int_{\mathbb{R}^n}\frac{{(f\left ( x \right )-h\left ( y \right ))} _{+} ^{p}}{{|x-y|}^{n+ps}  }dxdy=
	n\widetilde{V}_{-ps}(B^n,{\Pi}_{p,+}^{*,s}(f,h))\ge
	n{\omega} _{n}^{\frac{n+ps}{n}} |{\Pi }_{p,+}^{*,s}(f,h)|^{-\frac{ps}{n}}.
\end{eqnarray}
In the above inequality (\ref{7.3}), there is equality if $f$, $h$ are radially symmetric.

For the second inequality, by Lemma \ref{1.3}, 
		\begin{equation}\label{e6.6}
			|{\Pi }_{p,+}^{*,s}(f,h)|^{-ps/n}\ge |{\Pi }_{p,+}^{*,s}(f^{\star},h^{\star})|^{-ps/n},
		\end{equation}
		with equality if $f= f^{\star }(\phi x+x_{0} )$, $h= h^{\star }(\phi x+x_{0} )$ for some $\phi \in GL(n)$ and $x_{0}\in \mathbb{R}^n$. Since $f^{\star}$, $h^{\star}$ is radially
		symmetric, ${\Pi}_{p,+}^{*,s}(f^{\star},h^{\star})$ is a ball. Hence, it follows from (\ref{4.b}) and dual mixed volume inequality (\ref{2a}) that
		\begin{equation}\label{7.2}
			n{\omega} _{n}^{\frac{n+ps}{n}} |{\Pi }_{p,+}^{*,s}(f^{\star},h^{\star})|^{-\frac{ps}{n}}=
			\int_{\mathbb{R}^n}\int_{\mathbb{R}^n}\frac{{(f^{\star}\left ( x \right )-h^{\star}\left ( y \right ))} _{+} ^{p}}{{|x-y|}^{n+ps}}dxdy.   
		\end{equation}
		By combining (\ref{e6.6}) and (\ref{7.2}), the second inequality in the theorem is obtained.
	\end{proof}
	\noindent{\bf Proof of Theorem \ref{1.a}}
		For the first inequality, we set $K = B^n$ in (\ref{3b}) and apply the dual mixed volume
		inequality (\ref{2a}) to obtain
		\begin{eqnarray}\label{e6.8}
			\int_{\mathbb{R}^n}\int_{\mathbb{R}^n}\frac{{(f\left ( x \right )-h\left ( y \right ))} ^{p}}{{|x-y|}^{n+ps}  }dxdy=
			n\widetilde{V}_{-ps}(B^n,{\Pi}_{p}^{*,s}(f,h))\ge
			n{\omega} _{n}^{\frac{n+ps}{n}} |{\Pi }_{p}^{*,s}(f,h)|^{-\frac{ps}{n}}.
		\end{eqnarray}
		In the above inequality (\ref{e6.8}), there is equality if $f$, $h$ are radially symmetric.
		
		For the second inequality, by Lemma \ref{6.2}, 
		\begin{equation}\label{7.1}
			|{\Pi }_{p}^{*,s}(f,h)|^{-ps/n}\ge |{\Pi }_{p}^{*,s}(f^{\star},h^{\star})|^{-ps/n},
		\end{equation}
		with equality if $f= f^{\star }(\phi x+x_{0} )$, $h= h^{\star }(\phi x+x_{0} )$ for some $\phi \in GL(n)$ and $x_{0}\in \mathbb{R}^n$. Since $f^{\star}$, $h^{\star}$ is radially
		symmetric, ${\Pi}_{p}^{*,s}(f^{\star},h^{\star})$ is a ball. Hence, it follows from (\ref{3b}) and dual mixed volume inequality (\ref{2a}) that
		\begin{equation}\label{e6.10}
			n{\omega} _{n}^{\frac{n+ps}{n}} |{\Pi }_{p}^{*,s}(f^{\star},h^{\star})|^{-\frac{ps}{n}}=
			\int_{\mathbb{R}^n}\int_{\mathbb{R}^n}\frac{{(f^{\star}\left ( x \right )-h^{\star}\left ( y \right ))} ^{p}}{{|x-y|}^{n+ps}}dxdy.   
		\end{equation}
		By combining (\ref{7.1}) and (\ref{e6.10}), the second inequality in the theorem is concluded.\qed

$\;$\\
\noindent{\bf No competing interest is declared.}

%



\begin{thebibliography}{10}
	
\bibitem{Adams03}
R. A. Adams, J. J.  Fournie,
\newblock {\em Sobolev spaces},
Vol. 140. Elsevier, 2003.
		
\bibitem{AG89}
F. J. Almgren, E. H. Lieb,
\newblock Symmetric decreasing rearrangement is sometimes continuous,
\newblock{\em J. Amer. Math. Soc.}
2 (1989), 683-773.
		
\bibitem{TA76}
T. Aubin,
\newblock Problemes isoperimetriques et espaces de Sobolev,
\newblock{\em J. Differ. Geom.}
11 (1976) 573–598.
		
\bibitem{JHP01}
J. Bourgain, H. Brezis, P. Mironescu,
\newblock Another look at Sobolev spaces, In: Menaldi, J.L., Rofman, E., Sulem, A. (eds.). 
\newblock Optimal Control and Partial Differential Equations. A volume in honor of A. Bensoussans’s 60th birthday, Amsterdam: IOS Press; Tokyo: Ohmsha, 2001.
		 
\bibitem{BLL74}
H. J. Brascamp, E.H. Lieb, J.M. Luttinger,
\newblock A general rearrangement inequality for multiple integrals,
\newblock{\em J. Funct. Anal.}
17 (1974), 227-237.
		 
\bibitem{Burchard96}
A. Burchard,
\newblock Cases of equality in the Riesz rearrangement inequality,
\newblock{\em Ann. of Math.}
143 (1996), 499–527.
		
\bibitem{C00}
A. Cianchi,
\newblock Second-order derivatives and rearrangements,
\newblock{\em Duke Math. J.}
105 (2000), 355–385.
	   
\bibitem{CF06}
A. Cianchi, N. Fusco,
\newblock Minimal rearrangements, strict convexity and minimal points,
\newblock{\em Appl. Anal.}
85 (2006), 67–85.
	   
\bibitem{CLYZ09} 
A. Cianchi, E. Lutwak, D. Yang, G. Zhang,
\newblock Affine Moser–Trudinger and Morrey–Sobolev inequalities,
\newblock{\em Calculus of Variations and Partial Differential Equations},
36 (2009), 419–436.
	   
\bibitem{FV04}
A. Ferone, R. Volpicelli,
\newblock Convex symmetrization: the equality case in the P\'olya-Szeg\"o inequality,
\newblock{\em Calc. Var. Part. Diff. Equ.}
21 (2004), 259–272.
	   
\bibitem{Gardner06}
R. J. Gardner,
\newblock {\em Geometric Tomography, second edition},
\newblock Cambridge University Press, New York, 2006.

\bibitem{GZ98}
R. J. Gardner, G. Zhang,
\newblock Affine inequalities and radial mean bodies,
\newblock{\em Amer. J. Math.}
120(1998), 505–528.
       
\bibitem{CF09}
C. Haberl, F. E. Schuster,
\newblock General $L_p$ affine isoperimetric inequalities,
\newblock{\em J. Differ Geom.}
83 (2009), 1–26.
       
\bibitem{CF09-}
C. Haberl, F. E. Schuster,
\newblock Asymmetric affine $L_p$ Sobolev inequalities,
\newblock{\em J. Funct. Anal.}
257 (2009), 641–658.
	   
\bibitem{CF12}
C. Haberl, F. E. Schuster, J. Xiao,
\newblock An asymmetric affine P\'olya-Szeg\"o principle,
\newblock{\em Math. Ann.}
352 (2012), 517–542.
	   
\bibitem{HJM16}
J. Haddad, C. H. Jim\'enez, M. Montenegro,
\newblock Sharp affine Sobolev type inequalities via the $L_p$ Busemann-Petty centroid inequality,
\newblock{\em J. Funct. Anal.}
271 (2016), 454–473.
	   
\bibitem{JM24}
J. Haddad, M. Ludwig,
\newblock Affine fractional $L_p$ Sobolev inequalities,
\newblock{\em Math. Ann.}
388 (2024), 1091–1115.
	    
\bibitem{JM25}
J. Haddad, M. Ludwig,
\newblock Affine fractional  Sobolev and isoperimetric inequalities,
\newblock{\em J. Differ. Geom.}
129 (2025), 695-724. 

\bibitem{K21}
A. Kreuml,
\newblock The anisotropic fractional isoperimetric problem with respect to unconditional unit balls,
\newblock{\em Commun. Pure Appl. Anal.}
20 (2021), 783–799.
       
\bibitem{Lin17}
Y. Lin,
\newblock Affine Orlicz P\'olya-Szeg\"o principle for log-concave functions,
\newblock{\em J. Funct. Anal.}
273 (2017), 3295-3326.
       
\bibitem{Ludwig14}
M. Ludwig,
\newblock Anisotropic fractional Sobolev norms,
\newblock{\em Adv. Math.}
252 (2014), 150-157.
       
\bibitem{Ludwig1496}
M. Ludwig,
\newblock Anisotropic fractional perimeters,
\newblock{\em J. Differ. Geom.}
96 (2014), 77–93.
	   	
\bibitem{Lutwak75} 
E. Luwak,
\newblock Dual mixed volumes,
\newblock{\em Pacific J. Math.}
58  (1975), 531-538.
	   	
\bibitem{Lutwak00}
E. Luwak,
\newblock $L_p$ affine isoperimetric inequalities,
\newblock{\em J. Differ. Geom.}
56 (2000), 111–132.
	   	
\bibitem{LYZ02}
E. Lutwak, D. Yang, G. Zhang,
\newblock Sharp affine $L^p$ Sobolev inequalities,
\newblock{\em J. Differ. Geom.}
62 (2002), 17-38 .
	   
\bibitem{Ma14}
D. Ma,
\newblock Asymmetric anisotropic fractional Sobolev norms,
\newblock{\em Arch. Math. (Basel)}
103 (2014), 167–175.
		
\bibitem{Mazya11}
V. G. Maz'ya,
\newblock{\em Sobolev Spaces with Applications to Elliptic Partial Differential Equations},
Spring, Heidelberg, 2011.

\bibitem{N16}
V. H. Nguyen,
\newblock New approach to the affine P\'olya-Szeg\"o principle and the stability version of the affine Sobolev inequality,
\newblock{\em Adv. Math.}
302 (2016), 1080–1110.
		
\bibitem{PS51}
G. P\'olya, G. Szeg\"{o},
\newblock {\em Isoperimetric inequalities in Mathematical Physics, Ann. Math. Stud.}
27, Princeton University Press, Princeton, 1951.
		
\bibitem{Schneider14}
R. Schneider,
\newblock{\em Convex Bodies: The Brunn-Minkowski Theory},
Encyclopedia Math. Appl., vol. 151, Cambridge Univ. Press, Cambridge, 2014.
		
\bibitem{GT94}
G. Talenti,
\newblock Inequalities in rearrangement invariant function spaces, in: M. Krbec, A. Kufner, B. Opic, J. Rákosnik (Eds.), 
\newblock{\em Nonlinear Analysis, Function Spaces and Applications, vol. 5}
Prometheus, Prague, (1994), 177–230.

\bibitem{Wang12}
T. Wang,
\newblock  The affine Sobolev-Zhang inequality on $BV(\mathbb{R}^n)$.
\newblock{\em Adv. Math.}
230 (2012), 2457–2473.

\bibitem{Wang13}
T. Wang,
\newblock The affine P\'olya-Szeg\"o principle: equality cases and stability,
\newblock{\em J. Funct. Anal.}
265 (2013), 1728–1748.
		
\bibitem{Zhang99}
G. Zhang,
\newblock The affine Sobolev inequality,
\newblock{\em J. Differ. Geom.}
53 (1999), 183-202 .
		
\end{thebibliography}

\end{document}